\documentclass{amsart}

\usepackage{amsmath,latexsym,amssymb,amsthm,graphicx}
\usepackage{hyperref,doi}
\usepackage{stmaryrd}
\usepackage[all]{xy}
\usepackage{enumitem}
\setenumerate[1]{label=(\thesection.\arabic*)}

\numberwithin{equation}{section}

\newtheorem{theorem}{Theorem}[section]

\newtheorem{claim}[theorem]{Claim}
\newtheorem{corollary}[theorem]{Corollary}
\newtheorem{case}{Case}

\theoremstyle{definition}
\newtheorem{remark}[theorem]{Remark}
\newtheorem{remarks}[theorem]{Remarks}
\newtheorem{example}[theorem]{Example}
\newtheorem{definition}[theorem]{Definition}
\newtheorem{question}[theorem]{Question}
\newtheorem{conjecture}[theorem]{Conjecture}

\newcommand{\tn}[1]{\textnormal{#1}}
\newcommand{\sqb}[1]{\left[#1\right]}
\newcommand{\wt}[1]{\widetilde{#1}}

\def\Z{\mathbb{Z}}
\def\R{\mathbb{R}}
\def\N{\mathbb{N}}
\def\M{{M/\hspace{-3pt}\sim}}
\def\X{{X/\hspace{-3pt}\sim}}
\def\B{\mathcal{B}}
\def\BB{\B/\hspace{-2pt}\equiv}
\def\Mt{\wt{M}}
\def\Mtq{\Mt/\hspace{-3pt}\approx}
\def\int{\tn{Int}}
\def\s{\sigma}
\def\e{\varepsilon}
\def\A{{A/\hspace{-3pt}\sim}}

\def\ss{\vskip 1.5 pt plus 1pt}

\font\cuf=cmtt8
\newcommand{\curl}[1]{{\cuf #1}}

\def \Item#1!{\noindent\hbox to 18pt{\bf\kern2pt #1 \hss\ignorespaces}}

\begin{document}
\title{Orbit Spaces of Gradient Vector Fields}

\author[J.S.~Calcut]{Jack S. Calcut}
\address{Department of Mathematics\\
         Oberlin College\\
         Oberlin, OH 44074}
\email{jcalcut@oberlin.edu}
\urladdr{\href{http://www.oberlin.edu/faculty/jcalcut/}{\curl{http://www.oberlin.edu/faculty/jcalcut/}}}

\author[R.E.~Gompf]{Robert E. Gompf}
\address{Department of Mathematics\\
                    University of Texas at Austin\hfill\break
                    \indent 1 University Station C1200\\
                    Austin, TX 78712-0257}
\email{gompf@math.utexas.edu}
\urladdr{\href{http://www.ma.utexas.edu/users/gompf/}{\curl{http://www.ma.utexas.edu/users/gompf/}}}

\thanks{The second author was partially supported by NSF grant
DMS-1005304.}
\date{November 6, 2011}

\begin{abstract}
We study orbit spaces of generalized gradient vector fields for Morse functions.
Typically, these orbit spaces are non-Hausdorff.
Nevertheless, they are quite structured topologically and are amenable to study. We show that these orbit spaces are locally contractible.
We also show that the quotient map associated to each such orbit space is a weak homotopy equivalence and has the path lifting property.
\end{abstract}

\maketitle

\section{Introduction}\label{introduction}

Our interest in the topology of orbit spaces of smooth flows was sparked by the opening question in V.I.~Arnol{\textprime}d's \emph{Problems on singularities and dynamical systems}~\cite{arnold}:
\emph{can one write down explicitly a polynomial or trigonometric polynomial vector field on $\mathbb{R}^5$
whose orbit space is an exotic $\mathbb{R}^4$?} As noted by Arnol{\textprime}d, every exotic $\mathbb{R}^4$ appears in this
way for some smooth vector field on $\mathbb{R}^5$.
A curious corollary of Arnol{\textprime}d's observation is that there exists an uncountable collection of smooth vector fields on $\R^5$ each pair of which are topologically equivalent but not smoothly equivalent.
Motivated by Arnol{\textprime}d's problem, one is naturally led to consider the topology of orbit spaces for general smooth vector fields and 
potential implications for dynamics.
\ss

In~\cite{cgm}, the authors and J.~McCarthy gave an obstruction to manifolds appearing as orbit spaces of flows.
Namely, if the orbit space is merely semilocally simply-connected, then the induced homomorphism of fundamental groups is surjective.
In a positive direction, we gave a quadratic polynomial vector field on $\mathbb{R}^{2n-1}$ whose orbit space is
diffeomorphic to $\mathbb{C}P^{n-1}$ (naturally, this vector field comes from the complex Hopf fibration).
In a subsequent paper~\cite{cg}, the authors will give a general construction for producing infinitely many positive examples
of smooth manifolds appearing as orbit spaces of smooth vector fields on $\mathbb{R}^n$. For instance, smooth, noncompact, simply-connected $4$-manifolds realizing all possible intersection forms arise as orbit spaces of $\R^5$.\ss

Orbit spaces of general flows often have undesirable topological properties (e.g., non-Hausdorff, non-metrizable).
Nevertheless, they are ubiquitous in dynamics and are significant since their topological types are invariants of their associated flows up to topological equivalence.
The theory of $\R$-actions obviously differs markedly from that of actions of compact Lie groups, the latter having, for instance, closed and proper orbit space quotient maps~\cite[p.~38]{bredon}.
We find orbit spaces of $\R$-actions to be important and interesting in their own right, with their own problems, and aim to show that they are amenable to topological study. The present paper arose from our desire to show that a class of dynamically important flows (i.e., gradient flows for Morse functions) have semilocally simply-connected orbit spaces. For these orbit spaces, we obtain much more.\ss

Herein, we work with the class of $f$-gradient vector fields (defined in Section~\ref{definitions}) where $f$ is a Morse function.
These vector fields were introduced by A.V.~Pajitnov and include as proper subsets both gradient-like and Riemannian gradient vector fields for $f$~\cite[pp.~52--53]{pajitnov}.
In Theorem~\ref{loc_contr}, we show that orbit spaces of \hbox{$f$-gradients} are locally contractible
(see Theorem~\ref{loc_contr} for a more detailed statement).\ss

Let $\M$ be the orbit space of an $f$-gradient and let $q:M\to\M$ be the associated quotient map. Combining Theorem~\ref{loc_contr} with a theorem of McCord~\cite[Thm.~6]{mccord} (which extends work of Dold and Thom~\cite{doldthom}), we obtain in Corollary~\ref{weak_equiv} that $q$ is a weak homotopy equivalence (i.e., $q$ induces isomorphisms on all corresponding homotopy groups, and hence also on all corresponding singular (co)homology groups). As fibers of $q$ are contractible, this implies that $q$ is a \emph{quasifibration} (see Remarks~\ref{weak_equiv_remarks}). This raises the question of whether $q$ is a \emph{Serre fibration} (i.e., has the homotopy lifting property for disks $D^k$). In a positive direction, we prove in Theorem~\ref{pathlifting} that $q$ has the path lifting property for any specified lifts of both endpoints.\ss

For a compact Lie group $G$ acting topologically on a Hausdorff space $X$, Montgomery and Yang showed that the associated orbit space quotient map \hbox{$\pi:X\to\X$} has the path lifting property~\cite[Cor.~1]{montgomery_yang} (see also~\cite[Thm.~II.6.2]{bredon}). Furthermore, in case $X$ is also path-connected and locally simply-connected, and $G$ is connected, they showed that $\pi$ induces a surjection of fundamental groups and $\X$ is locally simply-connected~\cite[Cor.~2]{montgomery_yang} (see also~\cite[pp.~91--92]{bredon}).\ss

Orbit spaces of certain nice actions of non-compact groups on manifolds have been studied in, e.g.,~\cite{goel_neumann,mccann,nakayama,rainer,palais}.
The gradient flows studied in the present paper are not covered by any of these results.
For example, when $f$ has nonempty critical set, the flow associated to an $f$-gradient is not proper and the orbit space $\M$ is non-Hausdorff.\ss

Recall that Conner~\cite{conner} conjectured, and Oliver~\cite{oliver} later proved, that the orbit space of any continuous action of a compact Lie group on $\R^n$ is contractible. Clearly, Conner's conjecture cannot extend to general $\R$-actions by the Hopf examples mentioned above.
The results of the present paper raise the question of whether Conner's conjecture holds for orbit spaces of $f$-gradients on $\R^n$.\ss

It has been suggested to us by D.~Sullivan and A.~Verjovsky (independently) that manifolds of dimension 3 and 4 might be studied by looking at their $f$-gradient orbit spaces, which have dimension 2 and 3 respectively. To our knowledge, this interesting idea remains unexplored.\ss

This paper is organized as follows. Section~\ref{definitions} describes our conventions. Section~\ref{main_results} proves our main results. Section~\ref{examples} closes with a few pertinent examples and questions for further study.

\section{Definitions and Preliminaries}\label{definitions}

Throughout this paper, $M$ is a smooth (= $C^{\infty}$), connected, possibly noncompact $m$-manifold ($m\geq1$) without boundary, and $f:M\to\R$ is an arbitrary but fixed (smooth) Morse function. Every such $M$ admits a Morse function $f$~\cite[\S6]{milnor63}. Let $S(f)$ denote the set of critical points of $f$, which is a discrete subset of $M$. By our convention, manifolds are Hausdorff with a countable basis, and a \emph{map} is a continuous function.\ss

\begin{definition}[Pajitnov~\cite{pajitnov03}]\label{fgrad_def}
Let $v:M\to TM$ be a smooth vector field. Recall that $v(f):M\to \R$ is the smooth function defined by $v(f)(p):=df_{p}(v(p))$. Then, $v$ is an \hbox{\emph{$f$-gradient}} provided the following two conditions are satisfied:
\begin{enumerate}\setcounter{enumi}{\value{equation}}
\item\label{f_incr_orbits} $v(f)>0$ on $M-S(f)$.
\item\label{nondegen} Each $p\in S(f)$ is a non-degenerate minimum of $v(f)$.\setcounter{equation}{\value{enumi}}
\end{enumerate}
\end{definition}

\begin{remark}\label{clarify}
If $p\in S(f)$, then $df_{p}=0$ and $v(f)(p)=0$. (If $v$ is an \hbox{$f$-gradient} and $p\in S(f)$, then $v(p)=0$ as well~\cite[p.~50]{pajitnov}.) So, condition~\ref{f_incr_orbits} implies that $v(f):M\to[0,\infty)$ and each $p\in S(f)$ is a local minimum and a critical point of $v(f)$. Therefore, once~\ref{f_incr_orbits} is verified, the crux of~\ref{nondegen} is non-degeneracy.
\end{remark}

\begin{remark}
Utilizing the Hartman-Grobman theorem, Pajitnov proved that in a neighborhood of each \hbox{$p\in S(f)$} the flow generated by an $f$-gradient is topologically (but not necessarily diffeomorphically) flow equivalent to a standard hyperbolic flow \hbox{\cite[pp.~76--80]{pajitnov}} (see also Figure~\ref{morse_field} below).
\end{remark}

\begin{remark}
Every pair $(M,f)$ admits \hbox{$f$-gradients}, even when $M$ is noncompact. For instance, it is not difficult to extend the argument in~\cite[pp.~20--21]{milnor65} to noncompact $M$ to show that each pair $(M,f)$ admits a gradient-like vector field.
\end{remark}

Let $v:M\to TM$ be an arbitrary smooth vector field. If $p\in M$, then we let $\sqb{p}$ denote the image in $M$ of the maximal integral curve of $v$ through $p$; we call $\sqb{p}$ an \emph{orbit} of $v$.
Let $\sim$ denote the equivalence relation on $M$ whose classes are the orbits of $v$.
Let $\M$ denote the \emph{orbit space} of $v$ and let:
\[
	q:M\to\M
\]
denote the associated quotient map (by definition, $U\subset\M$ is open if and only if $q^{-1}(U)\subset M$ is open).
We say $v$ is \emph{complete} when the unique maximal flow generated by $v$ is defined for all $p\in M$ and $t\in\R$; this flow, denoted
\hbox{$\Phi:\R\times M\to M$}, is then a $1$-parameter group of diffeomorphisms of $M$ (see \hbox{\cite[\S2]{milnor63}} or \cite[\S6.2]{hirsch}).
When $M$ is compact, hence a closed manifold, every $v$ is complete, but this is not the case for noncompact $M$.
For our purposes, it will be convenient to have a complete vector field. This dilemma is easily overcome by scaling a given $f$-gradient by a suitable smooth function $M\to\R^{+}$ as we now explain.

\begin{claim}\label{scale_f-gradient}
Let $v$ be an $f$-gradient and let $\lambda:M\to\R^{+}$ be any smooth, positive function.
Then, the product vector field $\lambda v:M\to TM$ is an $f$-gradient.
Furthermore, $\sim$, $\M$, and $q$ are all unchanged upon replacing $v$ with $\lambda v$.
Lastly, there exists $\lambda:M\to\R^{+}$ so that the vector field $\lambda v$ is complete.
\end{claim}

\begin{proof}
First, we show $\lambda v$ is an $f$-gradient.
Clearly condition~\ref{f_incr_orbits} holds for $\lambda v$ since $(\lambda v)(f)=\lambda v(f)$.
To verify~\ref{nondegen} for $\lambda v$, let $p\in S(f)$. By Remark~\ref{clarify}, $p$ is a critical point of $(\lambda v)(f)$, which we must show is non-degenerate. By~\cite[\S6.1]{hirsch} or~\cite[\S2]{milnor63}, we may work locally and assume $M$ is an open subset of $\R^m$. Then, let $H$ denote the Hessian matrix of $v(f)$ at $p$, which is nonsingular since $v$ is an $f$-gradient. Straightforward computation shows that the Hessian matrix of $(\lambda v)(f)$ at $p$ equals $\lambda(p)H$, which is evidently nonsingular. So, $\lambda v$ is an $f$-gradient.\ss

Next, existence and uniqueness of solutions to first-order ordinary differential equations show that the maximal integral curves of $\lambda v$ are reparameterizations of those of $v$. Hence, $\sim$, and therefore $\M$ and $q$, are identical for $v$ and $\lambda v$.\ss

Finally, let $g$ be a complete Riemannian metric on $M$~\cite{nomizu_ozeki}. 
Using a smooth partition of unity, construct $\lambda:M\to\R^{+}$ so that $\lambda v$ has bounded norm on $M$. The vector field $\lambda v$ is complete~\cite[\S6.2]{hirsch}.
\end{proof}

Throughout Section~\ref{main_results},
$v$ denotes a fixed $f$-gradient which is complete.
By Claim~\ref{scale_f-gradient}, the results of Section~\ref{main_results} all hold true for arbitrary $f$-gradients.
The $q$-saturation of any set $U\subset M$ is $q^{-1}(q(U))=\cup_{t\in\R}\Phi_{t}(U)$. Thus, $q$ is an open map and $\M$ is path-connected, locally path-connected, and has a countable basis.
For clarity, we state three facts concerning topological separation properties of $\M$:
\begin{enumerate}\setcounter{enumi}{\value{equation}}
\item\label{T_0} $\M$ is $T_0$ (i.e., Kolmogorov).
\item\label{T_1} $\M$ is $T_1$ if and only if $S(f)=\emptyset$.
\item\label{T_2} $S(f)\neq\emptyset$ implies $\M$ is not $T_2$ (i.e., Hausdorff), although $\M$ may fail to be $T_2$ even when $S(f)=\emptyset$ (see Example~\ref{line_2_origins} in \S~\ref{examples}).\setcounter{equation}{\value{enumi}}
\end{enumerate}
These facts are easy consequences of the proof of Theorem~\ref{loc_contr} in the next section.

\section{Main Results}\label{main_results}

\begin{theorem}\label{loc_contr}
The orbit space $\M$ is locally contractible.
In fact, each point \hbox{$\sqb{p}\in\M$} has a countable neighborhood basis $\mathcal{N}\sqb{p}$ of open subsets of $\M$ such that if $Z\in\mathcal{N}\sqb{p}$, then $Z$ strong deformation retracts to $\sqb{p}$ and $q^{-1}(Z)$ is homeomorphic to $\R^m$.
\end{theorem}

\begin{proof}
There are two cases to consider, namely $p\in M-S(f)$ and $p\in S(f)$. First, let $p\in M-S(f)$. We will show that a neighborhood of $\sqb{p}$ in $\M$ is homeomorphic to $\R^{m-1}$. Condition~\ref{f_incr_orbits} implies that $f$ is strictly increasing along $\sqb{p}$. (In particular, there are no nonconstant periodic orbits.) By Sard's theorem, there are regular values of $f$ arbitrarily close to $f(p)$. So, replacing $p$ with a nearby point in $\sqb{p}$ if necessary, we may assume that $f(p)$ is a regular value of $f$. Thus, $f^{-1}(f(p))$ is a smooth submanifold of $M$ containing $p$ and transverse to $v$. Parameterize a neighborhood of $p$ in $f^{-1}(f(p))$ by a smooth embedding:
\[
	h:\R^{m-1}\hookrightarrow f^{-1}(f(p))\subset M
\]
such that $h(0)=p$. Define the local diffeomorphism:
\begin{equation}\begin{split}\label{H_diffeo}
\xymatrix@R=0pt{
      \R\times\R^{m-1}         \ar[r]^-{H} &   M   \\
    (t,x)   \ar@{|-{>}}[r]          &   \Phi_t(h(x))}
\end{split}
\end{equation}
Suppose $H(s,x)=H(t,y)$. Then $\Phi_{s-t}(h(x)))=\Phi_0(h(y))$ and the three properties: (i) $\Phi_0=\tn{id}_{M}$, (ii) $f$ is constant on $\tn{Im}(h)$, and (iii) $f$ is strictly increasing along each nonconstant orbit, imply that $s-t=0$. As $h$ is bijective, we get $x=y$ and $H$ is injective. Hence, $H$ is a diffeomorphism onto its image denoted $U(p)\subset M$. Note that $U(p)$ is an open, \hbox{$q$-saturated} neighborhood of $\sqb{p}$ in $M$, $q\left(U(p)\right)$ is an open neighborhood of $\sqb{p}$ in $\M$, and \hbox{$q|:U(p)\to q\left(U(p)\right)$} is an open map. Consider the following diagram where $\tn{pr}_2(t,x)=x$ is projection:
\begin{equation}\begin{split}\label{H_diagram}
\xymatrix{
        \R\times\R^{m-1}  \ar[r]^-{H} \ar[d]_{\tn{pr}_2}    &   U(p)   \ar[d]^{q|}\\
        \R^{m-1} \ar@{-->}[r]^-{\eta}   &  q\left(U(p)\right)}
\end{split}\end{equation}
The unique function $\eta$ that makes the diagram commute is a bijection. Two applications of the universal property of quotient maps show that $\eta$ is a homeomorphism. The required countable neighborhood basis of $\sqb{p}$ is now obtained by restricting diagram~\eqref{H_diagram} to open disks in $\mathbb{R}^{m-1}$ of radii $1/n$ for $n\in\mathbb{N}$.
This completes the first case.\ss

Second, let $p\in S(f)$. Then, $v(p)=0$ and $\sqb{p}=\{p\}$.
Let $k\in\{0,1,\ldots,m\}$ denote the index of $f$ at $p$.
We introduce some notation.
Let $\omega$ denote the vector field on $\R^m$ defined by:
\[
\omega(x)=(-x_1,\ldots,-x_k,x_{k+1},\ldots,x_m)
\]
which generates the standard hyperbolic flow:
\begin{equation}\label{omega_flow}
	\Omega_{t}(x)=\left(x_{1}e^{-t},\ldots,x_{k}e^{-t},x_{k+1}e^{t},\ldots,x_{m}e^{t}\right).
\end{equation}
Clearly $\Omega$ is a $1$-parameter group of diffeomorphisms of $\R^m$.
For each $R>0$, define the closed disks of radius $R$:
\begin{align*}
 D^{k}(R)&=\left\{x\in\R^k \mid x_1^2+\cdots+x_k^2\leq R^2\right\}\\
 D^{m-k}(R)&=\left\{x\in\R^{m-k}\mid x_{k+1}^2+\cdots+x_{m}^2\leq R^2\right\}
\end{align*}
and the open bidisk:
\begin{equation}\label{open_bidisk}
	\B(R)=\int D^{k}(R) \times \int D^{m-k}(R)\subset\R^{m}.
\end{equation}
Also, let $\equiv_R$ denote the equivalence relation on $\B(R)$ whose classes are the nonempty intersections of orbits of $\omega$ with $\B(R)$ (see Figure~\ref{morse_field}).
\begin{figure}[h!]
    \centerline{\includegraphics[scale=0.7]{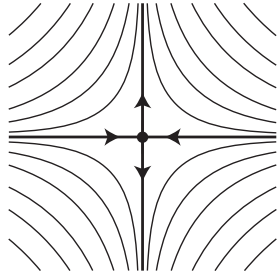}}
    \caption{Open bidisk $\B(R)\subset\R^m$ partitioned by its intersection with orbits of $\omega$ (cases $0<k<m$ pictured).}
    \label{morse_field}
\end{figure}

Given a manifold $N$, a smooth vector field $u:N\to TN$, and $z\in N$, we let $J(z,u)\subset\R$ denote the maximal open interval of definition of the maximal integral curve of $u$ that passes through $z$ at time $t=0$.\ss

As $v$ is an $f$-gradient and $p\in S(f)$, $\Phi$ is topologically flow equivalent to $\Omega$ near $p$~\cite[pp.~76--80]{pajitnov}.
More specifically, there exists an open neighborhood $U$ of $p$ in $M$, an $R>0$, and a homeomorphism $\psi:\mathcal{B}(R)\to U$ such that:
\begin{enumerate}\setcounter{enumi}{\value{equation}}
\item $\psi(0)=p$.
\item\label{int_def} $J(x,\left.\omega\right|{\B(R)})=J(\psi(x),\left.v\right|{U})$ for each $x\in \B(R)$.
\item\label{flow_equiv} $\Phi_{t}\circ\psi(x) = \psi\circ\Omega_{t}(x)$ for each $x\in \B(R)$ and 
$t\in J(x,\left.\omega\right|{\B(R)})$.
\setcounter{equation}{\value{enumi}}
\end{enumerate}

While $\psi$ captures the topological structure of $\Phi$ near $p$, the spaces $\mathcal{B}(R)/\hspace{-2pt}\equiv_R$ and $q(U)$ need \emph{not} be homeomorphic; $q(U)$ may have additional identifications since orbits of $v$ may well leave $U$ and then re-enter $U$ in forward time. This problem will be overcome by reducing $R>0$. Note that if $r\in(0,R]$, then the restriction homeomorphism $\left.\psi\right|:\mathcal{B}(r)\to\psi(\B(r))$ is again a flow equivalence from $\left.\omega\right|$ to $\left. v\right|$.

\begin{claim}\label{equiv_claim}
There is $r\in(0,R)$ so that for all $x,y\in\B(r)$: $x\equiv_r y$ $\Leftrightarrow$ $\psi(x)\sim\psi(y)$.
\end{claim}

\begin{proof}[Proof of Claim~\ref{equiv_claim}]
By~\ref{int_def} and~\ref{flow_equiv}, ``$\Rightarrow$'' holds for all $r\in(0,R]$, as does ``$\Leftarrow$'' in case $k=0$ or $k=m$.
So, assume $0<k<m$ and let $R_{0}\in(0,R)$ be arbitrary.
For $\rho_{1}\in(0,R_0)$, consider the closed bidisk $D^{k}(\rho_1)\times D^{m-k}(R_0)$ (i.e., the tall rectangle in Figure~\ref{lockers}).
\begin{figure}[h!]
    \centerline{\includegraphics[scale=0.7]{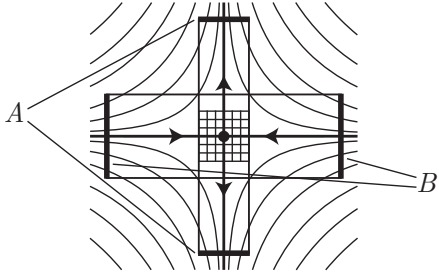}}
    \caption{Closed subsets $A$ and $B$ of $\B(R)$ on which $f\circ\psi>f(p)$ and $f\circ\psi<f(p)$ respectively. The open bidisk $\B(r)$ is hatched.}
    \label{lockers}
\end{figure}
By continuity, there is $\rho_{1}\in(0,R_0)$ small enough so that $f\circ \psi>f(p)$ on the closed set
$A:=D^{k}(\rho_1)\times\partial D^{m-k}(R_0)$ (see Figure~\ref{lockers}, and see \hbox{\cite[pp.~81--83]{pajitnov}} for a proof).
Similarly, there is $\rho_{2}\in(0,R_0)$ so that $f\circ\psi<f(p)$ on
$B:=\partial D^{k}(R_0)\times D^{m-k}(\rho_2)$ (see Figure~\ref{lockers}).
Let $r:=\min\{\rho_1,\rho_2\}$.
Suppose $\psi(x)\sim\psi(y)$ and $x\not\equiv_r y$ for some $x,y\in\B(r)$.
Swapping $x$ and $y$ if necessary, we have $\psi(y)=\Phi_{t}(\psi(x))$ for some $t>0$. Define:
\[
	J:=J(x,\left.\omega\right|{\B(r)})=J(\psi(x),\left.v\right|{\psi(\B(r))}).
\]
Then $t>\sup J$. Further, 
it is clear from Figure~\ref{lockers} and readily verified that there exist $0<t_1<t_2<t$ such that:
\[
	z_1 := \psi^{-1}\circ\Phi_{t_1}(\psi(x))\in A \quad \tn{and} \quad 
	z_2 := \psi^{-1}\circ\Phi_{t_2}(\psi(x))\in B.
\]
So, $f(z_1)>f(p)$ and $f(z_2)<f(p)$. But this contradicts the fact that $f$ is strictly increasing along $\sqb{\psi(x)}$ (see~\ref{f_incr_orbits}) and completes the proof of the claim.
\end{proof}

Let $r(p)$ denote the $r>0$ given by Claim~\ref{equiv_claim}. For each $0<r'\le r(p)$, let $\B(p,r')$ denote the open bidisk $\B(r')\subset\B(r(p))$ (see~\eqref{open_bidisk}) and note that Claim~\ref{equiv_claim} holds true with $r$ replaced by $r'$ (this notation will be reused in the proof of Theorem~\ref{pathlifting} below). For the remainder of the present proof, let $\B$ denote $\B(p,r')$ and let $\equiv$ denote $\equiv_{r'}$ for some arbitrary $0<r'\le r(p)$.
Consider the diagram:
\[
\xymatrix{
        \B  \ar[r]^-{\psi|} \ar[d]_{\pi}    &   \psi(\B)  \ar[d]^{q|}\\
        \BB \ar@{-->}[r]^-{\eta}   &  q(\psi(\B))}
\]
Here, $\pi$ is the quotient map associated to $\equiv$, $\psi|$ is the restriction homeomorphism, and $q|$ denotes the restriction map $\psi(\B)\to q(\psi(\B))$, which is an open map since $q$ is open and $\psi(\B)$ is open in $M$. In particular, $q(\psi(\B))$ is an open neighborhood of $\sqb{p}$ in $\M$. The unique function $\eta$ that makes the diagram commute is clearly surjective, and by Claim~\ref{equiv_claim} it is injective. Bicontinuity of $\eta$ follows by the universal property of quotient maps and since $q|$ is open. Thus, $\eta$ is a homeomorphism.\ss

We now show $\BB$ strong deformation retracts to $\pi(0)$. Clearly $0$ is a strong deformation retract of $\B$ by the homotopy:
\begin{equation}\begin{split}\label{G_homotopy}
\xymatrix@R=0pt{
      \B\times I         \ar[r]^-{G} &   \B   \\
    (x,s)   \ar@{|-{>}}[r]          &   sx}
\end{split}
\end{equation}
By~\eqref{omega_flow}, each $G_s$ sends each class of $\equiv$ into a class of $\equiv$. Consider the diagram:
\[
\xymatrix{
        \B\times I  \ar[r]^-{G} \ar[d]_{\pi\times\tn{id}}    &   \B  \ar[d]^{\pi}\\
        (\BB)\times I \ar@{-->}[r]^-{\mu}   &  \BB}
\]
As $I$ is locally compact and Hausdorff, and $\pi$ is a quotient map, we see that \hbox{$\pi\times\tn{id}$} is a quotient map. The universal property of quotient maps implies that $\mu$ is continuous, which is our desired homotopy.\ss

The last two paragraphs show that $q(\psi(\B))$ is an open neighborhood of $\sqb{p}$ in $\M$ that strong deformation retracts to $\sqb{p}$.
Recall that the $q$-saturation of \hbox{$\psi(\B)\subset M$} is:
\[
	q^{-1}(q(\psi(\B)))=\bigcup_{t\in\R}\Phi_t(\psi(\B)).
\]

\begin{claim}\label{unions_homeo}
The space $\cup_{t\in\R} \Omega_t(\B)$ is diffeomorphic to $\R^m$ and is homeomorphic to the \hbox{$q$-saturation} of  $\psi(\B)$.
\end{claim}
\begin{proof}[Proof of Claim~\ref{unions_homeo}]
The first conclusion is left as an exercise.
Next, define:
\begin{equation}\begin{split}\label{sat_homeo}
\xymatrix@R=0pt{
      \bigcup_{t\in\R} \Omega_t(\B)         \ar[r]^-{\Psi} &   \bigcup_{t\in\R}\Phi_t(\psi(\B))   \\
    	y   \ar@{|-{>}}[r]          &   \Phi_t(\psi(x))}
\end{split}
\end{equation}
where $y=\Omega_{t}(x)$ for some (nonunique) $t\in\R$ and $x\in\B$.
Well-definition and injectivity of $\Psi$ are pleasant exercises using~\ref{flow_equiv} and Claim~\ref{equiv_claim} respectively.
Surjectivity of $\Psi$ is immediate.
Another pleasant exercise shows that:
\begin{enumerate}\setcounter{enumi}{\value{equation}}
\item\label{flow_equiv_2} $\Phi_{s}\circ\Psi(y) = \Psi\circ\Omega_{s}(y)$ for each $y\in\bigcup_{t\in\R} \Omega_t(\B)$ and $s\in\R$.
\setcounter{equation}{\value{enumi}}
\end{enumerate}
Note that $\left.\Psi\right|{\B}=\psi$ and so~\ref{flow_equiv_2} extends~\ref{flow_equiv}.
As $\psi$ is a homeomorphism, property~\ref{flow_equiv_2} implies that $\Psi$ is a local homeomorphism.
As $\Psi$ is a bijection, we get that $\Psi$ is a homeomorphism (in fact, a flow equivalence by~\ref{flow_equiv_2}).
\end{proof}

The required countable neighborhood basis of $\sqb{p}$ is now obtained by taking $\B(p,r(p)/n)$ for $n\in\N$.
This completes the second case and the proof of Theorem~\ref{loc_contr}.
\end{proof}

Recall that a \emph{weak homotopy equivalence} is a map $h:X\to Y$ of topological spaces that induces isomorphisms:
\[
	h_{\sharp} : \pi_{i}(X,x) \to \pi_{i}(Y,h(x))
\]
for all $x\in X$ and all $i\ge 0$ (when $i=0$, isomorphism means bijection). The following theorem of McCord, which extends work of Dold and Thom~\cite[Satz~2.2]{doldthom}, says that weak homotopy equivalence may be checked locally.\ss

\begin{theorem}[McCord~{\cite[Thm.~6]{mccord}}]\label{mccord_thm}
Let $h:X\to Y$ be a map of topological spaces. Suppose $\mathcal{U}$ is an open cover of $Y$ satisfying:
\begin{enumerate}[label=\tn{(\alph*)}]\setcounter{enumi}{0}
\item\label{basis_like} If $U$ and $V$ lie in $\mathcal{U}$ and $y\in U\cap V$, then there exists $W\in\mathcal{U}$ such that $y\in W\subset U\cap V$.
\item\label{local_equiv} If $U\in\mathcal{U}$, then the restriction $\left.h\right|h^{-1}(U):h^{-1}(U)\to U$ is a weak homotopy equivalence.
\end{enumerate}
Then, $h$ itself is a weak homotopy equivalence.
\end{theorem}

The statements of Theorems~\ref{loc_contr} and~\ref{mccord_thm} immediately yield the following.\ss

\begin{corollary}\label{weak_equiv}
The orbit space quotient map $q:M\to\M$ is a weak homotopy equivalence.
\end{corollary}

\begin{remarks}\label{weak_equiv_remarks}
\Item(1)! As a consequence of Whitehead's theorem, each weak homotopy equivalence \hbox{$h:X\to Y$} also induces isomorphisms:
\[
	h_{\ast}:H_{i}(X;G)\to H_{i}(Y;G) \quad \textnormal{and} \quad h^{\ast}:H^{i}(Y;G)\to H^{i}(X;G)
\]
for all $i\ge 0$ and all coefficient groups $G$~\cite[p.~365]{hatcher}.

\Item(2)! Surjectivity of $q_{\sharp}:\pi_{1}(M,p)\to\pi_{1}(\M,\sqb{p})$ may be deduced in several alternative ways, all utilizing Theorem~\ref{loc_contr}. Namely, by~\cite[Thm.~1.1]{cgm}, Armstrong's~~\cite[Prop.~3(e)]{armstrong_82}, or by Theorem~\ref{pathlifting} below.

\Item(3)! In an earlier version of this paper, injectivity of $q_{\sharp}:\pi_{1}(M,p)\to\pi_{1}(\M,\sqb{p})$ was deduced by constructing the universal cover of $\M$~\cite{cgv1} (recall that classical covering space theory does not require any Hausdorff hypothesis~\cite[\S~1.3]{hatcher}). We briefly recall this construction since it may be of independent interest. Let $c:\Mt\to M$ be the (smooth) universal covering of $M$. Direct calculation shows that $f\circ c:\Mt\to\R$ is a Morse function. The vector field $v$ pulls back to $\wt{v}:\Mt\to T\Mt$ and $\wt{v}$ is an $(f\circ c)$-gradient. Thus, the flow $\Phi$ lifts (smoothly) to $\Mt$. The quotient $\Mtq$ of this lifted flow turns out to be the universal cover of $\M$ which quickly shows that $q_{\sharp}$ is injective on fundamental groups.

\Item(4)! As fibers of $q$ are contractible, Corollary~\ref{weak_equiv} implies that $q$ is a \emph{quasifibration}~\cite[p.~479]{hatcher} (see also~\cite{doldthom,mccord}).
\end{remarks}

The last remark raises the question of whether $q:M\to\M$ is a \emph{Serre fibration} (i.e., has the homotopy lifting property for disks $D^k$)~\cite[p.~376]{hatcher}. The following theorem is a partial positive answer, where we may further specify arbitrary lifts of both endpoints.

\begin{theorem}\label{pathlifting}
The map $q:M\to \M$ has path lifting. More specifically, let $\alpha:[0,1]\to\M$ be continuous, 
$e_{0}\in q^{-1}(\alpha(0))$, and $e_{1}\in q^{-1}(\alpha(1))$. Then, there exists a continuous function $\wt{\alpha}:[0,1]\to M$ such that $q\circ\wt{\alpha}=\alpha$, $\wt{\alpha}(0)=e_0$, and $\wt{\alpha}(1)=e_1$.
\end{theorem}

\begin{proof}
We begin with some preliminary observations.
Recall that: \[M=S(f)\sqcup (M-S(f))\] where $S(f)\subset M$ is closed, discrete, and $q$-saturated,
and $M-S(f)\subset M$ is open and $q$-saturated.
Define $\Sigma:=q(S(f))$, so: \[\M=\Sigma\sqcup q(M-S(f)).\]
Recall that $q$ is an open map, so $q(M-S(f))$ is open in $\M$ and $\Sigma$ is closed in $\M$.
If $p\in S(f)$, then $\sqb{p}=\{p\}$, $\sqb{p}$ is closed in $\M$, and $\sqb{p}$ is open (and closed) in $\Sigma$.
In particular, $\Sigma$ is discrete in $\M$. As $[0,1]$ is compact, $\tn{Im}(\alpha)\cap\Sigma$ is finite (possibly empty).
Also, $\alpha^{-1}(\Sigma)\subset[0,1]$ is closed and compact.\ss

Let $L$ denote the set of limit points of $\alpha^{-1}(\Sigma)$ in $[0,1]$.
Then, $L\subset\alpha^{-1}(\Sigma)$ and $L$ is closed in $[0,1]$.
We prove Theorem~\ref{pathlifting} in three successively more subtle cases: $\alpha^{-1}(\Sigma)=\emptyset$, $\alpha^{-1}(\Sigma)\neq\emptyset$ and $L=\emptyset$, and $L\neq\emptyset$.\ss

In any case, observe that $\wt{\alpha}$ is rigidly prescribed on the closed set $\alpha^{-1}(\Sigma)\subset[0,1]$.
Namely, $q|:S(f)\to\Sigma$ is a homeomorphism ($q|$ is open since $\Sigma$ is discrete in $\M$).
So, we must define:
\begin{equation}\begin{split}\label{lift_on_inv_Sigma}
\xymatrix@R=0pt{
     \alpha^{-1}(\Sigma)          \ar[r]^-{\wt{\alpha}} & S(f)\subset M  \\
    s   \ar@{|-{>}}[r]          &  q|^{-1}\circ\alpha(s) }
\end{split}\end{equation}
which is continuous.
It remains to lift $\alpha$ appropriately on the complement:
\[
	[0,1]-\alpha^{-1}(\Sigma) = \alpha^{-1}((\M)-\Sigma).
\]
This complement is the disjoint union of at most countably many open, connected subsets of $[0,1]$.\ss

The next claim will be our workhorse for constructing lifts.
Recall from the proof of Theorem~\ref{loc_contr} above that if $y\in M-S(f)$ and $f(y)$ is a regular value of $f$, then associated to $y$ are maps $h$, $H$, and $\eta$.

\begin{claim}\label{basic_lift}
Let $\beta:[a,b]\to\M$ be continuous such that $\tn{Im}(\beta)\subset q\left(U(y)\right)$ for some $y\in M-S(f)$ where $f(y)$ is a regular value of $f$.
Let $\tau:[a,b]\to\R$ be any continuous function.
Define the continuous function:
\begin{equation}\begin{split}\label{beta_lift}
\xymatrix@R=0pt{
     [a,b]          \ar[r]^-{\wt{\beta}} &  U(y)\subset M  \\
    s   \ar@{|-{>}}[r]          &  H\left(\tau(s),\eta^{-1}\circ\beta(s)\right) }
\end{split}\end{equation}
Then, $\wt{\beta}$ is a lift of $\beta$ (i.e., $q\circ\wt{\beta}=\beta$).
\end{claim}

\begin{proof}[Proof of Claim~\ref{basic_lift}]
Immediate by the diagram~\eqref{H_diagram}.
\end{proof}

\begin{remark}
Every continuous lift of $\beta$ to $M$ arises as in Claim~\ref{basic_lift} for some continuous $\tau$.
If $g$ is a continuous lift of $\beta$, then consider the map \hbox{$\tau=\tn{pr}_1\circ H^{-1}\circ g$} where $\tn{pr}_1:\R\times\R^{m-1}\to\R$ is projection.
\end{remark}

The simplest $\tau$, namely linear functions, provide just enough control over lifts to prove the first main case of Theorem~\ref{pathlifting} as we now show.

\begin{claim}\label{basic_endpoint_control}
Let $\beta:[a,b]\to\M$ be continuous such that $\tn{Im}(\beta)\subset q\left(U(y)\right)$ for some $y\in M-S(f)$ where $f(y)$ is a regular value of $f$.
Suppose we are given $e_a\in q^{-1}(\beta(a))$, $e_b\in q^{-1}(\beta(b))$, or both.
Then, there exists $\wt{\beta}$ a continuous lift of $\beta$ such that $\wt{\beta}(a)=e_a$, $\wt{\beta}(b)=e_b$, or both (respectively).
\end{claim}

\begin{proof}[Proof of Claim~\ref{basic_endpoint_control}]
Let $(t_a,x_a):=H^{-1}(e_a)$ and $(t_b,x_b):=H^{-1}(e_b)$.
If both endpoints are specified, then define $\tau:[a,b]\to\R$ to be linear such that $\tau(a)=t_a$ and $\tau(b)=t_b$; now, apply Claim~\ref{basic_lift}. If only one endpoint is specified, then reason similarly with $\tau$ constant.
\end{proof}

\begin{case}\label{alpha_inv_null}
Theorem~\ref{pathlifting} holds when $\alpha^{-1}(\Sigma)=\emptyset$.
\end{case}

\begin{proof}[Proof of Case~\ref{alpha_inv_null}]
Choose a finite subdivision $0=s_0<s_1<\cdots<s_N=1$ so that for each $0\leq i\leq N-1$ there exists
$y_i \in M-S(f)$ where $f(y_i)$ is a regular value of $f$ and:
\[
 \alpha([s_i,s_{i+1}]) \subset q\left(U(y_i)\right).
\]
Lift $\alpha|[s_0,s_1]$ using Claim~\ref{basic_endpoint_control} with $q^{-1}(\alpha(s_0))=e_0$ (the left endpoint) specified. Next, lift $\alpha|[s_1,s_2]$ using Claim~\ref{basic_endpoint_control} with the left endpoint specified to agree with the right endpoint of the previous lift. Repeat this process process until the last subinterval $[s_{N-1},s_{N}]$ where we further specify the right endpoint \hbox{$q^{-1}(\alpha(s_N))=e_1$}.
This completes the proof of Case~\ref{alpha_inv_null} of Theorem~\ref{pathlifting}.
\end{proof}

In Case~\ref{alpha_inv_null}, very little control of the lifts was required; we just had to match up endpoints.
When $\alpha^{-1}(\Sigma)\neq\emptyset$, more control is necessary. The next claim provides controlled lifts. We use the notation specified in the proof of Theorem~\ref{loc_contr} above immediately following the proof of Claim~\ref{equiv_claim}.

\begin{claim}\label{controlled_lift}
Let $\beta:[a,b]\to\M$ be continuous such that:
\begin{equation}\label{im_hypo}
	\tn{Im}(\beta)\subset q\left(U(y)\right) \cap q(\psi(\B(p,r')))
\end{equation}
for some $y\in M-S(f)$ where $f(y)$ is a regular value of $f$, some $p\in S(f)$, and some $0<r'\leq r(p)$.
Then, there exists $\wt{\beta}$ a continuous lift of $\beta$ such that $\tn{Im}(\wt{\beta})\subset\psi(\B(p,r'))$.
Further, given \hbox{$e_a\in q^{-1}(\beta(a))\cap \psi(\B(p,r'))$}, \hbox{$e_b\in q^{-1}(\beta(b))\cap \psi(\B(p,r'))$}, or both, we may choose $\wt{\beta}$ so that additionally \hbox{$\wt{\beta}(a)=e_a$}, \hbox{$\wt{\beta}(b)=e_b$}, or both.
\end{claim}

\begin{proof}[Proof of Claim~\ref{controlled_lift}]
Let $\mathcal{P}(\R)$ denote the power set of $\R$.
We define the \emph{carrier} $\phi:[a,b]\to\mathcal{P}(\R)$ by:
\[
	\phi(s):=\left\{t\in\R \mid H(t,\eta^{-1}\circ\beta(s))\in\psi(\B(p,r'))\right\}.
\]
Recall that $H(t,x)=\Phi_t\left(h(x)\right)$.
So, $\phi(s)$ equals the set of times for which the integral curve of $v$ beginning at $h\left(\eta^{-1}\circ\beta(s)\right)$ lies in the open set $\psi(\B(p,r'))$.
By~\eqref{im_hypo} and since $\psi$ is a flow equivalence (see the proof of Theorem~\ref{loc_contr}), $\phi(s)$ is a nonempty, connected open interval (possibly infinite on one side) for each $s\in[a,b]$.
We assert that $\phi$ is lower semi-continuous (see~\cite[p.~365]{michael} for a definition).
But, this is immediate by continuity of $\Phi$, openness of $\eta$, and continuity of $\beta$.
Thus, Michael's selection theorem~\cite[Thm.~3.1$'''$]{michael}
implies that $\phi$ admits a \emph{selection}. That is, there exists a continuous function $\tau:[a,b]\to\R$ such that $\tau(s)\in\phi(s)$ for every $s\in[a,b]$.
Now, apply Claim~\ref{basic_lift} using this $\tau$. The resulting lift $\wt{\beta}$ behaves as desired. If $e_a$ is specified, then alter $\phi$ by defining $\phi(a):=\{t_a\}$ where $e_a=H(t_a,\eta^{-1}(\beta(a)))$. It is easy to check that this modified $\phi$ is lower semi-continuous. Argue similarly in case $e_b$ is specified.
\end{proof}

\begin{remark}
The previous claim used only the relatively simple implication (a)$\Rightarrow$(b) of Michael's Theorem~3.1$'''$ (see~\cite[\S5]{michael}).
\end{remark}

Repeated applications of Claims~\ref{basic_endpoint_control} and~\ref{controlled_lift} yield the following.

\begin{claim}\label{one_endpoint_sing}
Let $\beta:[a,b]\to\M$ be continuous such that \hbox{$\beta^{-1}(\Sigma)=\{b\}$}.
Then, there exists $\wt{\beta}$ a continuous lift of $\beta$.
Further, given \hbox{$e_a\in q^{-1}(\beta(a))$}, we may choose $\wt{\beta}$ so that additionally $\wt{\beta}(a)=e_a$.
Also, the analogous results hold in case \hbox{$\beta^{-1}(\Sigma)=\{a\}$}.
\end{claim}

\begin{proof}[Proof of Claim~\ref{one_endpoint_sing}]
It suffices to prove the first two conclusions.
For the first conclusion, we assume, without loss of generality, that $a=0$ and $b=1$.
Choose a countable subdivision:
\[
	0=s_0<s_1<s_2<s_3<\cdots<1
\]
such that $s_i\rightarrow1$ and, for each $i$, $\beta([s_i,s_{i+1}])\subset q\left(U(y_i)\right)$ for some $y_i\in M-S(f)$ where $f(y_i)$ is a regular value of $f$.
As $\beta(1)\in\Sigma$, there is a unique $p\in S(f)$ such that $\beta(1)=\sqb{p}$.
Define $\wt{\beta}(1):=p$ (we have no choice here).\ss

Recursively define $N_1<N_2<N_3<\cdots$ an increasing sequence of positive integers as follows.
The set $q\left(\psi(\B(p,r(p)))\right)$ is an open neighborhood of $\sqb{p}$ in $\M$.
As $\beta$ is continuous, there exists $N_1\in\Z^{+}$ such that
$\beta([s_{N_1},1])\subset q\left(\psi(\B(p,r(p)))\right)$.
Having constructed $N_1<\cdots<N_k$, choose $N_{k+1}>N_k$ such that
$\beta([s_{N_{k+1}},1])\subset q\left(\psi(\B(p,r(p)/(k+1)))\right)$.\ss

Now, we define $\wt{\beta}$ on $[0,1)$ by successively lifting $\beta|[s_i,s_{i+1}]$ for $i=0,1,2,\ldots$ as follows.
For each $i>0$, we tacitly assume that the lift of $\beta|[s_i,s_{i+1}]$ will agree with that of $\beta|[s_{i-1},s_{i}]$ at the common endpoint $s_i$.
For $i=0,\ldots,N_1-1$, lift $\beta|[s_i,s_{i+1}]$ using 
Claim~\ref{basic_endpoint_control} with $y_i\in M-S(f)$;
for $i=N_1-1$, we further specify that the lift send
$s_{N_1}$ into $q\left(\psi(\B(p,r(p)))\right)$.
For $i=N_{1},\ldots,N_2-1$, lift $\beta|[s_i,s_{i+1}]$ using
Claim~\ref{controlled_lift} with $y=y_i$ and $r'=r(p)/1$;
for $i=N_2-1$, we further specify that the lift send
$s_{N_2}$ into $q\left(\psi(\B(p,r(p)/2))\right)$.
Assume that $\beta|[0,s_{N_k}]$ has been so lifted.
For $i=N_{k},\ldots,N_{k+1}-1$, lift $\beta|[s_i,s_{i+1}]$ using
Claim~\ref{controlled_lift} with $y=y_i$ and $r'=r(p)/k$;
for $i=N_{k+1}-1$, we further specify that the lift send
$s_{N_{k+1}}$ into $q\left(\psi(\B(p,r(p)/(k+1)))\right)$.
This completes our definition of $\wt{\beta}$.\ss

It is clear by construction that $\wt{\beta}$ is a lift of $\beta$ and is continuous on $[0,1)$. For continuity at $1$, let $W$ be any neighborhood of $p=\wt{\beta}(1)$ in $M$.
There is $k\in\Z^{+}$ such that $\psi(\B(p,r(p)/k))\subset W$. Then, $\wt{\beta}$ sends $(s_{N_k},1]$ into $W$ as desired.\ss

For the second conclusion, we simply specify in the very first lift ($i=0$) that $s_0$ is sent to $e_0$.
\end{proof}

\begin{claim}\label{both_endpoints_sing}
Let $\beta:[a,b]\to\M$ be continuous such that $\beta^{-1}(\Sigma)=\{a,b\}$.
Then, there exists $\wt{\beta}$ a continuous lift of $\beta$.
\end{claim}

\begin{proof}[Proof of Claim~\ref{both_endpoints_sing}]
Without loss of generality, $a=0$ and $b=1$. Apply Claim~\ref{one_endpoint_sing} to each half $[0,1/2]$ and $[1/2,1]$ taking care to match up the lifts at $1/2$.
\end{proof}

We now prove the second main case of Theorem~\ref{pathlifting}.

\begin{case}\label{L_null_ainv_non-null}
Theorem~\ref{pathlifting} holds when $\alpha^{-1}(\Sigma)\neq\emptyset$ and $L=\emptyset$.
\end{case}

\begin{proof}[Proof of Case~\ref{L_null_ainv_non-null}]
In this case, $\alpha^{-1}(\Sigma)\neq\emptyset$ is a finite set of points in $[0,1]$.
Let \hbox{$\s_1<\s_2<\cdots<\s_n$} be the points in $\alpha^{-1}(\Sigma)$.
If $\s_1\neq0$, then lift $\alpha|[0,\s_1]$ using Claim~\ref{one_endpoint_sing} with $e_0$ specified.
Successively lift $\alpha|[\s_i,\s_{i+1}]$ using Claim~\ref{both_endpoints_sing} (adjacent endpoints are automatically matched up). If $\s_n\neq1$, then lift $\alpha|[\s_n,1]$ using Claim~\ref{one_endpoint_sing} with $e_1$ specified.
This completes the proof of Case~\ref{L_null_ainv_non-null} of Theorem~\ref{pathlifting}.
\end{proof}

The third main case of Theorem~\ref{pathlifting} will use a sharpened version of Claim~\ref{both_endpoints_sing}. First, we sharpen Claim~\ref{one_endpoint_sing}.

\begin{claim}[Addendum to Claim~\ref{one_endpoint_sing}]\label{one_endpoint_sing_small}
Assume further that
$\tn{Im}(\beta)\subset q(\psi(\B(p,r')))$ for some (unique) $p\in S(f)$ and some $0<r'\leq r(p)$.
Then, in all three conclusions of Claim~\ref{one_endpoint_sing}, we may choose $\wt{\beta}$ so that $\tn{Im}(\wt{\beta})\subset\psi(\B(p,r'))$.
\end{claim}

\begin{remark}\label{p_unique}
Uniqueness of $p\in S(f)$ is a consequence of two facts:
(i) for each $z\in S(f)$, $q(\psi(\B(z,r(z))))$ intersects $\Sigma=q(S(f))$ in exactly the single point $\sqb{z}$, and
(ii) $\beta^{-1}(\Sigma)$ equals $\{b\}$ or $\{a\}$.
In particular, $p=q^{-1}(\beta(b))$ or $p=q^{-1}(\beta(a))$.
\end{remark}

\begin{proof}[Proof of Claim~\ref{one_endpoint_sing_small}]
The proof is very similar to the proof of Claim~\ref{one_endpoint_sing}, except now we choose a sequence
$0=N_1<N_2<N_3<\cdots$
of integers such that $\beta([s_{N_{k}},1])\subset q\left(\psi(\B(p,r'/k)))\right)$. And, every individual lift is provided by Claim~\ref{controlled_lift}. Details are left to the reader.
\end{proof}

\begin{claim}[Addendum to Claim~\ref{both_endpoints_sing}]\label{both_endpoints_sing_small}
Assume further that
$\tn{Im}(\beta)\subset q(\psi(\B(p,r')))$ for some (unique) $p\in S(f)$ and some $0<r'\leq r(p)$.
Then, we may choose $\wt{\beta}$ so that $\tn{Im}(\wt{\beta})\subset\psi(\B(p,r'))$.
\end{claim}

\begin{proof}[Proof of Claim~\ref{both_endpoints_sing_small}]
Identical to that of Claim~\ref{both_endpoints_sing}, except using Claim~\ref{one_endpoint_sing_small}.
\end{proof}

We now prove the third and final main case of Theorem~\ref{pathlifting}.

\begin{case}\label{L_non-null}
Theorem~\ref{pathlifting} holds when $L\neq\emptyset$.
\end{case}

\begin{proof}[Proof of Case~\ref{L_non-null}]
Recall that $\wt{\alpha}$ is already (rigidly) prescribed on $\alpha^{-1}(\Sigma)$.
Notice that it suffices to lift $\alpha$ on:
\[
	K:=[\inf \alpha^{-1}(\Sigma), \sup \alpha^{-1}(\Sigma)] \subset [0,1].
\]
Indeed, given a continuous lift of $\alpha|K$, suppose $0\not\in\alpha^{-1}(\Sigma)$.
Lift $\alpha|[0,\inf \alpha^{-1}(\Sigma)]$ using Claim~\ref{one_endpoint_sing} with $e_0$ specified.
These two lifts agree at their common endpoint, namely $\inf \alpha^{-1}(\Sigma)$, so they
yield a continuous lift on $[0,\sup \alpha^{-1}(\Sigma)]$.
Argue similarly if $1\not\in\alpha^{-1}(\Sigma)$.
So, it remains to lift $\alpha|K$.
Without loss of generality, $K=[0,1]$.\ss

Consider an arbitrary subinterval:
\begin{equation}\label{Vinterval}
	V:=[\s_{-},\s_{+}]\subset[0,1]
\end{equation}
where $\s_{-}<\s_{+}$ and $[\s_{-},\s_{+}]\cap\alpha^{-1}(\Sigma)=\{\s_{-},\s_{+}\}$.
We lift $\alpha|V$ according to three mutually exclusive cases:
\begin{enumerate}[label=\tn{(\alph*)}]\setcounter{enumi}{0}
\item\label{no_p} For every $p\in S(f)$, $\alpha(V)\not\subset q(\psi(\B(p,r(p))))$.
\item\label{p_k} For some (unique) $p\in S(f)$ and (unique) $k\in\Z^{+}$, $\alpha(V)\subset q(\psi(\B(p,r(p)/k)))$ and
								$\alpha(V)\not\subset q(\psi(\B(p,r(p)/(k+1))))$.
\item\label{p_all_k} For some (unique) $p\in S(f)$ and every $k\in\Z^{+}$, $\alpha(V)\subset q(\psi(\B(p,r(p)/k)))$.
\end{enumerate}
In cases~\ref{p_k} and~\ref{p_all_k}, uniqueness of $p\in S(f)$ follows as in Remark~\ref{p_unique}. In particular, $\alpha(\s_{-})=\sqb{p}$ and $\alpha(\s_{+})=\sqb{p}$.\ss

In case~\ref{no_p}, lift $\alpha|V$ using Claim~\ref{both_endpoints_sing}.
In case~\ref{p_k}, lift $\alpha|V$ using Claim~\ref{both_endpoints_sing_small} with $r'=r(p)/k$.
In case~\ref{p_all_k}, define the continuous function:
\[
\xymatrix@R=0pt{
     [0,1]          \ar[r]^-{\delta} & [0,1]  \\
    s   \ar@{|-{>}}[r]          &  \tn{dist}(s,L) }
\]
Let $l(V):=\s_{+}-\s_{-}$ denote the length of $V$. Recall that $L\neq\emptyset$ and notice that:
\[
	0<\frac{l(V)}{2} \leq \sup (\delta|V) \leq 1.
\]
Lift $\alpha|V$ using Claim~\ref{both_endpoints_sing_small} with $r':=(\sup (\delta|V))\cdot r(p)$.
This completes our definition of $\wt{\alpha}$.\ss

There exist at most countably many such subintervals $V$, and any two of them are either identical, disjoint, or they intersect at a single shared endpoint (which necessarily lies in $\alpha^{-1}(\Sigma)$).
Also, every $s\in[0,1]$ lies in either $\alpha^{-1}(\Sigma)$ or in the interior of a unique such $V$.
It follows that $\wt{\alpha}$ is a well-defined lift of $\alpha$.
It remains to verify that $\wt{\alpha}$ is continuous.\ss

Evidently, continuity of $\wt{\alpha}$ is only in question at points of $L$.
So, let $s_0\in L$.
We verify continuity of $\wt{\alpha}$ at $s_0$ from the left; continuity from the right is similar.
If $\alpha^{-1}(\Sigma)$ does not accumulate on $s_0$ from the left, then there is nothing to show.
So, assume otherwise.
Let $p\in S(f)$ be the unique point such that $\alpha(s_0)=\sqb{p}$.
By construction, $\wt{\alpha}(s_0)=p$.
The open sets $\psi(\B(p,r(p)/k))$, $k\in\Z^{+}$, form a countable basis at $p$ in $M$.
Thus, for each fixed  $k\in\Z^{+}$ it suffices to produce $\e>0$ such that \hbox{$\wt{\alpha}((s_0-\e,s_0])\subset\psi(\B(p,r(p)/k))$}.
So, fix $k\in\Z^{+}$.
As $q$ is open and $\alpha$ is continuous, the set:
\begin{equation}\label{Bknhbd}
	W:=\alpha^{-1}(q(\psi(\B(p,r(p)/k))))\subset[0,1]
\end{equation}
is an open neighborhood of $s_0$ in $[0,1]$.
By~\eqref{Bknhbd}, there exists $\e_0>0$ such that:
\begin{equation}\label{e0}
 (s_0-\e_0,s_0]\subset W \quad \tn{and} \quad \e_0<1/k.
\end{equation}
As $\alpha^{-1}(\Sigma)$ accumulates on $s_0$ from the left, there exists:
\begin{equation}\label{sigma}
	\s\in\alpha^{-1}(\Sigma)\cap(s_0-\e_0,s_0).
\end{equation}
Define $\e:=s_0-\s>0$.
We verify that this $\e$ works. Any point in $\alpha^{-1}(\Sigma)\cap(s_0-\e,s_0]$ lies in $W$ and, hence, maps by $\wt{\alpha}$ to $p$ itself.
Next, let $s\in(s_0-\e,s_0]-\alpha^{-1}(\Sigma)$. Then, $s$ lies in the interior of a unique closed interval $V$ as in~\eqref{Vinterval}.
By~\eqref{sigma} and~\eqref{e0}:
\begin{equation}\label{maincont}
	V\subset [s_0-\e,s_0] \subset (s_0-\e_0,s_0]\subset W.
\end{equation}
In case~\ref{p_k}, the lift provided by Claim~\ref{both_endpoints_sing_small} sends $V$ into $\psi(\B(p,r(p)/k))$, so we are done.
In case~\ref{p_all_k}, \eqref{maincont} implies that:
\[
	0<\sup (\delta|V) \leq \e <\e_0 <1/k.
\]
So, again the lift provided by Claim~\ref{both_endpoints_sing_small} sends $V$ into $\psi(\B(p,r(p)/k))$.
This completes the proof of Case~\ref{L_non-null} of Theorem~\ref{pathlifting}.
\end{proof}

The proof of Theorem~\ref{pathlifting} is complete.
\end{proof}

\begin{remark}
Claim~\ref{controlled_lift} has a (more precise) closed analogue. Here, the hypothesis $\tn{Im}(\beta)\subset q(\psi(\B(p,r')))$ is replaced with the hypothesis $\tn{Im}(\beta)\subset q\left(\psi\left(\overline{\B(p,r')}\right)\right)$, where $0<r'<r(p)$ and $\overline{\B(p,r')}$ is the closure of $\B(p,r')$ in $\B(p,r(p))$.
The conclusion becomes $\tn{Im}(\wt{\beta})\subset \psi\left(\overline{\B(p,r')}\right)$. The proof of this fact requires a bit more work, again employing Michael's selection theorem~\cite{michael}. We suppress this result since the open version of Claim~\ref{controlled_lift} is sufficient for our purposes.
With some tinkering, one may even avoid the use of Michael's selection theorem completely, but at present we see no advantage to such an approach.
\end{remark}

\section{Examples and Questions}\label{examples}

\begin{example}\label{sphere_ex}
Consider the standard longitudinal flow on $S^n$ as in Figure~\ref{sphere}.
\begin{figure}[h!]
    \centerline{\includegraphics{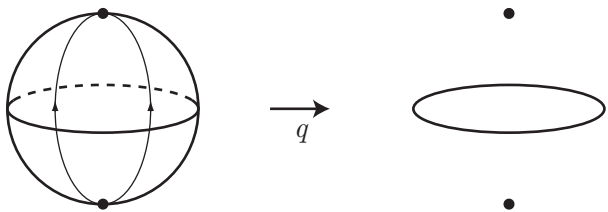}}
    \caption{Longitudinal flow on $S^n$ and its orbit space.}
    \label{sphere}
\end{figure}
Corollary~\ref{weak_equiv} implies that $q$ is a weak homotopy equivalence. The orbit space $q(S^n)$ is the \emph{non-Hausdorff suspension} of $S^{n-1}$~\cite[p.~472]{mccord}. A natural inductive argument yields a topological space $X_n$ containing $2n+2$ points and a weak homotopy equivalence $S^n\to X_n$; this result was first observed by McCord~\cite[\S8]{mccord} (cf. \cite[p.~38]{barmak}). Clearly there is no weak homotopy equivalence $X_n\to S^n$.
\end{example}

\begin{example}\label{line_2_origins}
We give an instance of the general situation studied in Section~\ref{main_results} and where $S(f)=\emptyset$, but $\M$ is non-Hausdorff.
Let $M=\R^2-\{(0,0)\}$, let $f:M\to\R$ be $f(x,y)=y$, and
let $v:M\to TM$ be the constant $(0,1)$ vector field.
Note that $f$ is a Morse function, $S(f)=\emptyset$, and $v$ is an $f$-gradient ($v$ is not complete, but this is immaterial in view of Claim~\ref{scale_f-gradient}).
The orbit space $\M$ is readily seen to be homeomorphic (diffeomorphic, in fact) to $\mathcal{L}$, the \emph{line with two origins}~\cite[pp.~15]{hirsch}.
Recall that $\mathcal{L}$ is the quotient of $\R\times\{1,2\}$ by the equivalence relation generated by $(x,1)\equiv(x,2)$ for each $x\neq0$; $\mathcal{L}$
is a smooth, non-Hausdorff $1$-manifold. By Corollary~\ref{weak_equiv}, $q:M\to\mathcal{L}$ is a weak homotopy equivalence. Any map from $\mathcal{L}$ to a Hausdorff space factors through $\R$, so there is no weak homotopy equivalence $\mathcal{L}\to M$.
\end{example}

\begin{remark}
It is instructive to compare and contrast Examples~\ref{sphere_ex} and~\ref{line_2_origins} with the orbit spaces of the standard hyperbolic flows on $\R^1$ and $\R^2$ of various indexes (see~\eqref{omega_flow}), all of which are contractible.
In particular, the index $1$ flow on $\R^2$, with the origin removed, is the double cover of $M$ from Example~\ref{line_2_origins}.
\end{remark}

\begin{example}\label{pi_1_observations}
For non-gradient vector fields, $q_{\sharp}$ is not necessarily injective on fundamental groups:
consider the obvious circle action on any product $M=N\times S^1$.
Or, if $v:M\to TM$ has a dense orbit $\sqb{p}$, then $\M$ strong deformation retracts to $\sqb{p}$ (the homotopy is the identity when $t=0$ and sends all points to $\sqb{p}$ for all $t\in(0,1]$).
Is $q_{\sharp}$ a surjection on fundamental groups for \emph{arbitrary} smooth vector fields (cf.~\cite{armstrong_82}, \cite{cgm})?
\end{example}

\begin{conjecture}\label{slsc_conj}
The orbit space associated to an \emph{arbitrary} smooth vector field on a Hausdorff manifold $M$, having a countable basis and empty boundary, is semilocally simply-connected.
\end{conjecture}

An affirmative answer to Conjecture~\ref{slsc_conj} implies an affirmative answer to the question in Example~\ref{pi_1_observations} by~\cite[Thm.~1.1]{cgm}.

\begin{example}
For general maps, path lifting, as in Theorem~\ref{pathlifting} above, is a relatively rare but special property.
For non-gradient vector fields, the associated quotient maps need not have the path lifting property.
Consider the closed annulus $A\subset\R^2$ with orbits as in Figure~\ref{spiral}.
\begin{figure}[h!]
    \centerline{\includegraphics{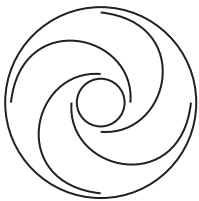}}
    \caption{Integral curves for a spiral vector field on a closed annulus $A\subset\R^2$.}
    \label{spiral}
\end{figure}
Let $q:A\to\A$ be the associated quotient map.
Let $X$ and $Y$ denote the inner and outer boundary components of $A$ respectively.
Define $x:=q(X)$ and $y:=q(Y)$.
Fix some point $z\in(\A)-\{x,y\}$.
The continuous path $\alpha:[0,1]\to\A$ defined by: $\alpha(0)=x$, $\alpha(1)=y$, and $\alpha(t)=z$ for $0<t<1$, has no continuous lift to $A$.
Similar examples on closed manifolds may be obtained either by doubling $A$ or, more generally, by embedding $A$ in a closed surface and then extending the smooth vector field on $A$ arbitrarily.
Still, the reader may verify that $\alpha$ is homotopic (rel endpoints) to a path which does admit a continuous lift to $A$.
Indeed, the map $q:A\to\A$ has path lifting \emph{up to homotopy} (rel endpoints).
Does every quotient map $q:M\to\M$ induced by a smooth, not necessarily gradient, vector field have path lifting up to homotopy (rel endpoints)? 
\end{example}

\begin{question}
The present paper focused on the topology of orbit spaces of \hbox{$f$-gradients}.
As a next step to understanding the topology of orbit spaces of arbitrary smooth vector fields,
we ask: what are the analogues of our results for circle-valued Morse functions, in particular for cellular gradients~\cite[Ch.~12]{pajitnov}?
\end{question}

\section*{Acknowledgement}
The authors thank Jason D. Mireles James for helpful conversations.

\end{document}